\documentclass{siamart0216}

\newcommand{\TheTitle}{Convex Geometry of the Generalized matrix-fractional Function} 
\newcommand{\TheAuthors}{J. V. Burke, Y. Gao, and T. Hoheisel}

\headers{\TheTitle}{\TheAuthors}

\title{{\TheTitle}}

\author{
	James V. Burke\thanks{Department of Mathematics, University of Washington, Seattle, WA 98195 (\email{jvburke@uw.edu}). Research is supported in part by the National Science Foundation under
	grant number DMS­-1514559.} 
	\and
	Yuan Gao\thanks{Department of Applied Mathematics, University of Washington, Seattle, WA 98195 (\email{yuangao@uw.edu}).}
	\and
	Tim Hoheisel\thanks{McGill University, 805 Sherbrooke St West, Room1114, Montr\'eal, Qu\'ebec, Canada H3A 0B9 \email{tim.hoheisel@mcgill.ca})}
}

\usepackage{lipsum}
\usepackage{graphicx}
\usepackage{amsmath,amsfonts,amssymb}
\ifpdf
  \DeclareGraphicsExtensions{.eps,.pdf,.png,.jpg}
\else
  \DeclareGraphicsExtensions{.eps}
\fi

\def\del{\delta}

\newcommand{\aff}{\mathrm{aff}\,}

\newcommand{\co}{\mathrm{conv}\,}
\newcommand{\cl}{\mathrm{cl}\,}
\newcommand{\clco}{\mathrm{\overline{conv}}\,}

\newcommand{\dom}{\mathrm{dom}\,}
\newcommand{\epi}{\mathrm{epi}\,}

\newcommand{\inter}{\mathrm{int}\,}
\newcommand{\rint}{\mathrm{rint}\,}

\newcommand{\rge}{\mathrm{rge}\,}

\newcommand{\lin}{\mathrm{span}\,}

\newcommand{\R}{\mathbb{R}}
\newcommand{\Rn}{\R^n}
\newcommand{\Rnm}{\R^{n\times m}}

\newcommand{\bS}{\mathbb{S}}
\newcommand{\Sn}{\bS^n}

\newcommand{\rank}{\mathrm{rank}\,}

\newcommand{\eR}{\mathbb R\cup\{+\infty\}}
\newcommand{\bR}{\mathbb{R}}
\def\Rn{\bR^n}

\def\Rnm{\bR^{n\times m}}
\def\bS{\mathbb{S}}
\def\Sn{\bS^n}
\def\Snp{\bS^n_+}
\def\Snpp{\bS_{++}^n}

\newcommand{\tr}{\mathrm{tr}\,}

\newcommand{\cone}{\mathrm{cone}\;}

\newcommand{\bN}{\mathbb{N}}
\newcommand{\bE}{\mathbb{E }}

\newcommand{\map}[3]{#1 :#2\rightarrow #3}
\newcommand{\ip}[2]{\left\langle #1,\, #2\right\rangle}
\newcommand{\half}{\frac{1}{2}}
\newcommand{\norm}[1]{\left\Vert #1\right\Vert}

\newcommand{\set}[2]{\left\{#1\,\left\vert\; #2\right.\right\}}
\newcommand{\ncone}[2]{N_{#2}\left(#1\right)}

\newcommand{\support}[2]{\sig_{#2}\left(#1\right)}

\newcommand{\indicator}[2]{\del_{#2}\left(#1\right)}
\newcommand{\gauge}[2]{\gam_{#2}\left(#1\right)}

\def\sd{\partial}

\newcommand{\bt}{{\bar{t}}}

\newcommand{\gam}{\gamma}

\newcommand{\sig}{\sigma}
\newcommand{\Sig}{\Sigma}
\newcommand{\lam}{\lambda}

\def\eps{\epsilon}

\newcommand{\cD}{\mathcal{D}}
\newcommand{\cE}{\mathcal{E}}

\newcommand{\cK}{\mathcal{K}}

\newcommand{\cS}{\mathcal{S}}

\newcommand{\hZ}{{\widehat{Z}}}

\newcommand{\AND}{\ \mbox{ and }\ }

\ifpdf
\hypersetup{
  pdftitle={\TheTitle},
  pdfauthor={\TheAuthors}
}
\fi

\begin{document}

\maketitle

\begin{abstract}
Generalized matrix-fractional (GMF) functions are a class of matrix
support functions introduced by Burke and Hoheisel
as a tool for unifying a range of seemingly divergent matrix optimization
problems associated with   
inverse problems, regularization and learning.
In this paper we dramatically simplify the support function representation for GMF functions
as well as the representation of their subdifferentials. 
These new representations allow the ready computation of a range of important 
related geometric objects whose formulations were previously unavailable.
\end{abstract}



\section{Introduction}
Generalized matrix-fractional (GMF) functions were introduced in \cite{Bur15}
as a means to unify a range of seemingly divergent tools in matrix
optimization related to inverse problems, regularization and
machine learning. Somewhat surprisingly GMF functions
coincide with the negative of the optimal value function for affinely constrained
quadratic programs, and are representable as
support functions on the matrix space $\bE:=\Rnm\times\Sn$,
where $\Rnm$ and $\Sn$ are the vector spaces of real $n\times m$ and symmetric
$n\times n$ matrices, respectively.
The most significant challenge in \cite{Bur15} is the derivation of an expression
for the closed convex set associated with the support function representation. 
Unfortunately, the representation given in \cite{Bur15}
is exceedingly complicated. The main contribution
of this paper is to provide a simple, elegant, and intuitive representation for this set.
We then use this representation to provide a simplified expression for the
subdifferential of a GMF function and to compute various related geometric objects
that were previously unavailable. These representations dramatically simplify
the use of these tools to a wide range of applications \cite{BGH-applications17}.
Before proceeding, we review the definition of a GMF function.

Given $(A,B)\in \R^{p\times n}\times\R^{p\times m}$ with
$\rge B\subset \rge A$, the graph of the matrix valued mapping 
$Y\mapsto -\half YY^T$ over an affine manifold $\set{Y\in\Rnm}{AY=B}$
is given by
\begin{equation}\label{eq:D}
\cD(A,B) :=\set{\left(Y,-\half YY^T\right)\in \bE}{Y\in \mathbb R^{n\times m}:\;AY=B}.
\end{equation}
The associated GMF function is the
support function of the set $\cD(A, B)$:
\[
\sigma_{ \cD (A,B)} (X,V) = \sup_{(Y,W)\in\cD(A,B)}\ip{(X,V)}{(Y,W)},
\]
where we use the Frobenius inner product on $\bE$, 
\[
\ip{(Y,W)}{(X,V)}=\tr(Y^TX)+\tr{WV}=\tr(XY^T+WV).
\]
In \cite[Theorem 4.1]{Bur15}, it is shown that
\begin{equation}\label{eq:support D}
\sigma_{ \cD (A,B)} (X,V) = 
\left\{\begin{array}{lcl}\frac{1}{2} \tr\!\left( \binom{X}{B}^TM(V)^\dagger \binom{X}{B}  \right) 
& {\rm if} \ \; \rge \binom{X}{B}\subset \rge M(V), \; V\in\cK_A ,\\
+\infty & {\rm else},
\end{array}\right.
\end{equation}
where $\cK_A:=\set{V\in\Sn}{u^TVu\ge 0\ (u\in \ker A)}$ and
$M(V)^\dagger$ is the Moore-Penrose pseudo inverse of the matrix
\[
M(V)=\begin{pmatrix}V&A^T\\ A&0\end{pmatrix}.
\]
In particular, this implies that
\begin{equation}\label{eq:dom support D}
\begin{aligned}
\dom \sigma_{D(A,B)}&=\dom \partial \sigma_{D(A,B)}\\
& =
\set{(X,V)\in \R^{n\times m}\times \Sn}{  \rge \binom{X}{B}\subset \rge M(V), \; 
V\in\cK_A   }.
\end{aligned}
\end{equation}
Note that $\dom \sigma_{D(A,B)}$ is clearly not a closed set.
To see this consider the case
$A=B=0$ and $V=\eta I$ so that 
any $X\ne 0$ has $\rge X\in\rge V$.
But as $\eta\downarrow 0$ it is not the case that $\rge X\subset \rge 0$.
Consequently, the statement in 
\cite[Theorem 4.1]{Bur15} that this domain is closed is clearly false.
This error does not affect the validity of the other results in \cite{Bur15}
since none of them require that the set $\dom \sigma_{D(A,B)}$
be closed.

The representation \eqref{eq:support D} is the basis for the name {\it generalized
matrix-fractional function} since the matrix-fractional functions 
\cite{BoV 04,Dat 05,Gal 11,HsO 14} are obtained when 
the matrices $A$ and $B$ are both taken to be zero.

The paper is organized as follows:
Section \ref{sec:omega} begins with a study of the cones $\cK_A$
defined in \eqref{eq:K} and their polars. This is immediately followed
by deriving our new representation of the set $\Omega(A,B):=\clco \cD(A,B)$ in
Theorem \ref{thm:clco D}. With this representation in hand,
we derive new simplified descriptions  for the normal cone  $N_{\Omega(A,B)}$
and the subdifferential $\sd \sig_{\Omega(A,B)}$ in Section \ref{sec:ncone}.
In Section \ref{sec:geometry} we explore the convex geometry of  the set
$\Omega(A,B)$, and conclude in Section \ref{sec:gauge} with the important
special case where $B=0$ and $\sig_{\Omega(A,0)}$ is a gauge function.
\\
\\
{\bf Notation:}  Let $\cE$ be a finite dimensional Euclidean space with inner product denoted by 
$\ip{\cdot}{\cdot}$ and the induced norm $\|\cdot\|:=\sqrt{\ip{\cdot}{\cdot}}$ 
with the closed
$\eps$-ball about a point $x\in\cE$ denoted by $B_\eps(x)$.  
%
Let   $S\subset\cE$ be nonempty.   The (topological) {\em closure} and {\em interior} of  $S$ are denoted by $\cl S$ and $\inter S$,  respectively.  The {\em (linear) span} of $S$  will be denoted by $\lin S$.

The {\em convex hull} of $S$ is the set of all convex combinations of elements of $S$ and is denoted by $\co S$.  Its closure (the {\em closed convex hull}) is $\clco S:=\cl(\co S)$.   The {\em conical  hull} of $S$ is the set 
\[
\R_+S:=\set{\lambda x}{x\in S,\;\lambda\geq 0}.
\]
The {\em convex conical hull} of $S$ is 
\[
\cone S:=\set{\sum_{i=1}^r\lambda_i x_i}{r\in \bN,\; x_i\in S,\;\lambda_i\geq 0}.
\]
It is easily seen that $\cone S=\R_+(\co S)=\co(\R_+ S)$. The closure of the latter is 
$
\overline \cone S:=\cl (\cone S).
$
The {\em affine hull} of $S$, denoted by $\aff S$,  is the smallest affine space that contains $S$. 

The {\em relative interior} of a convex set $C\subset \cE$ is its interior in the relative topology with respect to the affine hull, i.e.
\[
\rint C=\set{x\in C}{\exists \varepsilon>0: \;B_\varepsilon(x)\cap \aff C\subset C}.
\]
It is well known, see e.g. \cite[Section 6.2]{BaC11}, that the points $x\in \rint C$ are characterized through  
\begin{equation}\label{eq:RintChar}
\R_+(C-x)=\lin (C-x),
\end{equation}
where the latter is the (unique) subspace parallel to $\aff C$. In particular, we have $\R_+C=\aff C=\lin C$ if and only if $0\in \rint C$.

The {\em polar set} of $S$ is defined by 
\[
S^\circ:=\set{v\in \cE}{\ip{v}{x}\leq 1\;(x\in S)}.
\]
Moreover, we define the {\em bipolar  set} of $S$ by  $S^{\circ\circ}:=(S^{\circ})^\circ$. It is well known that $S^{\circ \circ}=\overline \cone(S\cup \{0\})$. 
If $K\subset \cE$ is  a  cone (i.e. $\R_+K\subset K$)  it can be seen by a homogeneity argument that  
\[
K^\circ =\set{v\in \cE}{\ip{v}{x}\leq 0\;(x\in K)},
\]
and if $\cS\subset\cE$ is a subspace, $\cS^\circ$ is the orthogonal subspace $\cS^\perp$.
The {\em horizon cone} of $S$ is the set 
\[
S^\infty:=\set{v\in \cE}{\exists \{\lambda_k\}\downarrow 0,\; \{x_k\in S\}:\; \lambda_k x_k\to v}
\]
which is always a closed cone. For a convex set $C\subset \cE$, $C^\infty$ 
coincides with the {\em recession cone} of the closure of $C$, i.e.
\[
C^\infty=\set{v}{x+tv\in \cl C\;(t\geq 0, \;x\in C)}
=\set{y}{C+y\subset C}.
\]


\noindent
For $\map{f}{\cE}{\eR}$ its {\em domain} and {\em epigraph} are given by
\[
\dom f:=\set{x\in\cE}{ f(x)<+\infty}\AND \epi f:=\set{(x,\alpha)\in\cE\times \R}{f(x)\leq \alpha}.
\]
We call $f$ {\em convex} if its epigraph $\epi f$ is a convex set. 

For a convex  function $\map{f}{\cE}{\eR}$ its  {\em subdifferential} at a point $\bar x\in \dom f$ is given by
\[
\partial f(\bar x):=\set{v\in \cE}{f(x)\geq f(\bar x)+\ip{v}{x-\bar x}}.
\]
Given a nonempty set $S\subset \cE$, its {\em indicator function} $\delta_S:\cE\to\eR$ is given by 
\[
\indicator{x}{S} :=\left\{\begin{array}{rcl}0 & \text{if} & x\in S,\\
+\infty &  \text{if} & x\notin S.
\end{array}\right.
\] 
The indicator of $S$ is convex if and only if  $S$ is a convex set, in which case the {\em normal cone} of $S$ at $\bar x\in S$ is given by
\[
\ncone{\bar x}{S}:=\partial \delta_S(\bar x)=\set{v\in \cE}{\ip{v}{x-\bar x}\leq 0 \;(x\in S)}.
\]
The  {\em support function} $\sigma_S:\cE\to \eR$ and the {\em gauge function} $\gamma_S:\cE\to\eR$ of a nonempty set $S\subset \cE$ are given by 
\[
 \support{x}{S}:=\sup_{v\in S}\ip{v}{x}\AND \gauge{x}{S}:=\inf\set{t \geq  0}{x\in tS},
\]
respectively. Here we use the standard convention that  $\inf \emptyset=+\infty$. 
It is easy to see that 
\begin{equation}\label{eq:support sd 1}
\sigma_S=\sigma_{\clco S}.
\end{equation}

\section{New Representation of {\boldmath$\clco\cD(A, B)$}}
\label{sec:omega}

\noindent
In view of \eqref{eq:support sd 1}, in order to obtain a complete understanding of the variational properties of $\sigma_S$, it is critical to have a useful description of the closed convex hull $\clco S$. This is often a non-trivial task.
In \cite[Proposition 4.3]{Bur15}, a representation for $\clco\cD(A,B)$ is
obtained after great effort, and this representation is arduous. Although it
is successfully used in \cite[Section 5]{Bur15} in several important situations,
the representation is an obstacle to a deeper understanding of the function
$\sigma_{ \cD (A,B)}$ as well as its ease of use in applications. The focus of this section
is to provide a new and intuitively appealing representation that dramatically
facilitates the use of $\sigma_{ \cD (A,B)}$. 
The key to this new representation is the class of cones 
\begin{equation}\label{eq:K}
\cK_\cS:=\set{V\in\Sn}{u^TVu \ge 0,\;(u \in \cS)},
\end{equation}
where $\cS$ is a subspace of $\Rn$, that is,
$\cK_\cS$
is the set of all symmetric matrices that are positive definite with respect to the given subspace $\cS$.
Observe that if $P\in\bS^n$ is the orthogonal projection onto $\cS$, then
\begin{equation}\label{eq:KAlt}
\cK_\cS=\set{V\in\bS^n}{PVP\geq 0}.
\end{equation}

\noindent
Clearly, $\cK_\cS$ is a convex cone, and, for $\cS=\R^n$, it reduces to $\bS_+^n$. 
Given a matrix $A\in \R^{p\times n}$, the cones $\cK_{\ker A}$ play a special role
in our analysis. For this reason, we simply write  $\cK_A$ to denote $\cK_{\ker A}$,
i.e. $\cK_A:=\cK_{\ker A}$.


\begin{proposition}[$\cK_\cS$ and its polar]\label{prop:Snp and polar}
Let $\cS$ be a nonempty subspace of $\Rn$ and let $P$ be the orthogonal
projection onto $\cS$. Then the following hold:
\begin{itemize}
\item[a)]
$
\cK_\cS^{\circ}=\cone \set{-vv^T}{v\in \cS}=\set {W\in \bS^n}{W=PWP\preceq  0}.
$
\item[b)]
$
\inter \cK_\cS= \set{V\in\Sn}{u^TVu> 0\;(u\in \cS\setminus\{0\})}.
$

\item[c)] 
$
\aff (\cK_\cS^\circ)=\lin \set{vv^T}{v\in \cS}= \set{W\in \bS^n}{\rge W\subset \cS}.
$

\item[d)]  
$
\rint(\cK_\cS^{\circ})=\set{W\in\cK_\cS^{\circ}}{u^TWu<0\ \;(u\in \cS\setminus\{0\})}
$
when $\cS\ne\{0\}$ and 

$\rint(\cK_{\{0\}}^{\circ})=\{0\}$ (since $\cK_{\{0\}}=\Sn$).

\end{itemize}
\end{proposition}

\begin{proof} \hfill\\
\begin{itemize}
\item[a)] Put $B:=\set{-ss^T}{s\in \cS}\subset\bS^n_-$ and observe that 
\[
\cone B=\set{-\sum_{i=1}^r\lambda_i s_is_i^T}{r\in \bN, s_i\in \cS,\lambda_i\geq 0\;(i=1,\dots,r)}.
\] 
We have $\cone B= \set{W\in \bS_-^n}{W=PWP}$: To see this, first note that
$\cone B \subset \set{W\in \bS_-^n}{W=PWP}$. The reverse inclusion invokes the spectral decomposition of $W=\sum_{i=1}^n\lambda_i q_iq_i^T$ for $\lambda_1,\dots,\lambda_n\leq  0$.  
In particular, this representation of $\cone B$ shows that it is closed.  
We now prove the first equality in a): To this end, observe that 
\begin{eqnarray*}
\cK_\cS & = & \set{V\in \bS^n}{s^TVs\geq 0\;(s\in \cS)}\\
& = & \set{V\in \bS^n}{\ip{V}{-ss^T}\leq 0\;(s\in \cS)}\\
& = & (\cone B)^\circ,
\end{eqnarray*}
where the third equality uses simply the  linearity of the inner product in the second argument. Polarization then gives 
\[
\cK_\cS^\circ=(\cone B)^{\circ\circ}=\overline\cone B=\cone B.
\]
 

\item[b)] The proof is straightforward and follows the pattern of proof
for $\inter \Snp=\Snpp$.

\item[c)] With $B$ as defined above, observe that
\[
\aff \cK_\cS^\circ =\lin \cK_\cS^\circ=\lin B,
\] 
since $0\in \cK_\cS^\circ$, which shows the first equality.  It is hence obvious that 
$\aff \cK_\cS\subset \set{W\in \bS^n}{\rge W\subset \cS}$. 
On the other hand, every $W\in \bS^n$ such that $\rge W\subset \cS$ 
has a decomposition 
$W=\sum_{i=1}^{\rank W}\lambda_i q_iq_i^T$ where $\lambda_i\neq 0$ and 
$q_i\in  \rge W\subset \cS$ for all $i=1,\dots,\rank W$, i.e. $W\in \lin B=\aff \cK_\cS^\circ$.

\item[d)] 
Set $R:=\set{W\in\cK_\cS^{\circ}}{u^TWu<0\ \;(u\in \cS\setminus\{0\})} $ and let 
$W\in \rint(\cK_\cS^\circ)\setminus R\subset \cK_\cS^\circ$. 
Then there exists $u\in \cS$ with  $\|u\|=1$ such that $u^TWu=0$. 
Then   for every $\varepsilon>0$ we have  $u^T(W+\varepsilon uu^T)u=\varepsilon>0$. 
Therefore $W+\varepsilon uu^T\in (B_\varepsilon(W)\cap 
\aff (\cK_\cS^\circ))\setminus \cK_\cS^\circ$ 
for all $\varepsilon>0$, and hence $W\notin \rint(\cK_\cS^\circ)$, which contradicts our 
assumption. Hence, $\rint(\cK_\cS^\circ)\subset R$.

To see the reverse implication assume there were $W\in R\setminus \rint(\cK_\cS^\circ)$, i.e.  for all $k\in \bN$ there exists $W_k\in B_{\frac{1}{k}}(W)\cap \aff 
(\cK_\cS^\circ)\setminus \cK_\cS^\circ.$   
In particular, there exists $\set{u_k\in \cS}{\norm{u_k}=1}$ such that $u_k^TW_ku_k\ge 0$ 
for all $k\in \bN$. W.l.o.g. we can assume that $u_k\to u\in \cS\setminus\{0\}$. 
Letting $k\to\infty$, we find that $u^TWu\geq 0$ since $W_k\to W$. This 
contradicts the fact that $W\in R$. 
\end{itemize}
\end{proof}

\noindent
We are now in a position to prove the main result of this paper which gives a new, simplified description of the closed convex hull of $\Omega(A,B)$.

\begin{theorem}\label{thm:clco D}
	Let $\cD(A, B)$ be as given by \cref{eq:D}, then $\clco \cD(A,B)=\Omega(A,B)$, where
	\begin{equation}\label{eq:omega}
	\Omega(A,B) :=\set{(Y,W)\in\bE}
	{AY=B \! \!\AND\!\! \half YY^T+W\in \cK_A^\circ}.	
	\end{equation}
\end{theorem}
\begin{proof}
We first show that $\Omega(A,B)$ is itself a closed convex set. Obviously,
$\Omega(A,B)$ is closed since $\cK_A^\circ$ is closed and the mappings $Y\mapsto AY$
and $(Y,W)\mapsto \half YY^T+W$ are continuous. 

So we need only show that $\Omega(A,B)$ is convex: To this end, let $(Y_i, W_i)\in \Omega(A,B),\ i=1,2$ and $0\le \lam\le 1$. Then there exist $M_i \in \cK_A^\circ, \ i=1,2$ such that $W_i = -\half Y_iY_i^T + M_i$. Observe that $A((1-\lam)Y_1+\lam Y_2)=B$. Moreover, we compute that
\[
\begin{aligned}
&\half ((1-\lam)Y_1+\lam Y_2)((1-\lam)Y_1+\lam Y_2)^T+((1-\lam)W_1+\lam W_2)\\
=&\half((1\!-\!\lam)Y_1\!+\!\lam Y_2)((1\!-\!\lam)Y_1\!+\!\lam Y_2)^T\!+\!\!\left(\!(1\!-\!\lam)(-\half Y_1Y_1^T\!+\!M_1)\!+\!\lam(-\half Y_2Y_2^T\!+\!M_2)\!\right)\\
=&\half\lam(1-\lam)(-Y_1Y_1^T+Y_1Y_2^T+Y_2Y_1^T-Y_2Y_2^T)+(1-\lam)M_1+\lam M_2\\
=&\lam(1-\lam)\left(-\half(Y_1 - Y_2)(Y_1 - Y_2)^T\right)+(1-\lam)M_1+\lam M_2.
\end{aligned}
\]
Since $\rge (Y_1 - Y_2) \subset\ker A$, this shows $\lam(1-\lam)\left(-\half(Y_1 - Y_2)(Y_1 - Y_2)^T\right)+(1-\lam)M_1+\lam M_2 \in \cK_A^\circ$. Consequently, $\Omega(A,B)$ is a closed convex set.

Next note that if $(Y,-\half YY^T)\in\cD(A,B)$, then
$(Y,-\half YY^T)\in\Omega(A,B)$ since $0\in \cK_A^\circ$. Hence, $\clco\cD(A,B)\subset\Omega(A,B)$.

It therefore remains to establish the reverse inclusion: For these purposes, let $(Y,W)\in \Omega(A,B)$. By Carath\'{e}odory's theorem, there exist $\mu_i \ge 0, v_i \in \ker A\; (i=1, \ldots, N)$  such that
\[
W = -\half YY^T - \sum_{i=1}^N\mu_iv_iv_i^T,
\]
where $N = \frac{n(n+1)}{2}+1$. Let $0<\eps<1$. Set $\lam_1 := 1-\eps$ and $\lam_2=\ldots=\lam_{N+1}=\lam:= \eps/N$. Denote $Y_1 := Y/\sqrt{1 - \eps}$. Take $Z_i \in \Rnm,\ i=1,\ldots,N$ such that $AZ_i = B$. Finally, set 
\[
V_i = \left[\sqrt{\frac{2\mu_i}{\lam}}v_i, 0, \ldots, 0\right]\in \R^{n\times m}\AND Y_{i+1} = Z_i + V_i,\; (i=1,\ldots, N).
\]
Observe that
\[
\sum_{i=1}^{N+1}\lam_iY_i = \sqrt{1-\eps} Y + \frac{\eps}{N}\sum_{i=2}^{N+1}Y_i = \sqrt{1-\eps} Y + \frac{\eps}{N}\sum_{i=1}^{N}Z_i + \sqrt{\frac{\eps}{N}}\sum_{i=1}^{N}\bar{V}_i,
\]
where $\bar{V}_i = [\sqrt{2\mu_i}v_i, 0, \ldots, 0],\ i=1,\ldots, N$, and 
\[
\begin{aligned}
-\half\sum_{i=1}^{N+1}\lam_iY_iY_i^T &= -\half YY^T - \half\sum_{i=1}^N\frac{\eps}{N}\left(Z_iZ_i^T + Z_iV_i^T + V_iZ_i^T\right) -\sum_{i=1}^N\mu_iv_iv_i^T\\
& = W - \sum_{i=1}^N\half\left(\frac{\eps}{N} Z_iZ_i^T + \sqrt{\frac{\eps}{N}}Z_i\bar{V}_i^T + \sqrt{\frac{\eps}{N}}\bar{V}_iZ_i^T\right),
\end{aligned}.
\]
Therefore  
{\small
\begin{gather}
\left(\sqrt{1-\eps} Y + \frac{\eps}{N}\sum_{i=1}^{N}Z_i + \sqrt{\frac{\eps}{N}}\sum_{i=1}^{N}\bar{V}_i,\quad  W - \sum_{i=1}^N\half\left(\frac{\eps}{N} Z_iZ_i^T + \sqrt{\frac{\eps}{N}}Z_i\bar{V}_i^T + \sqrt{\frac{\eps}{N}}\bar{V}_iZ_i^T\right)\right) \nonumber
\\ =\left(\sum_{i=1}^{N+1}\lam_iY_i,\quad  -\half\sum_{i=1}^{N+1}\lam_iY_iY_i^T\right).
\label{eq:eps limit}
\end{gather}
} 
Set $\kappa:=\dim \bE$. By Carath\'eodory's theorem,  
\[
\co \cD(A,B)\!=\!
\set{\!\!\left(\sum_{i=1}^{\kappa+1}\lam_iY_i,-\half\sum_{i=1}^{\kappa+1}\lam_iY_iY_i^T\right)\!\!}
{\!\begin{matrix}
	\lambda\in \R^{{\kappa+1}}_+, \sum_{i=1}^{\kappa +1}\lambda_i=1,\,  Y_i\in\R^{n\times m}\\
	\, \ AY_i=B\ (i=1,\dots,{\kappa+1)}
	\end{matrix}\!}.
\]
By letting $\eps\downarrow 0$ in \eqref{eq:eps limit}, we find $(Y,W)\in\clco\cD(A,B)$ thereby
concluding the proof.
\end{proof}

\section{Normal cone of {\boldmath${\Omega(A, B)}$ 
and the  subdifferential of $\sigma_{\cD(A,B)}$}}
\label{sec:ncone}
The  new representation for $\clco \cD(A,B)$ allows us to dramatically
simplify the representation for the subdifferential of $\sig_{\cD(A,B)}$ given in
\cite[Theorem 4.8]{Bur15}.  
For this we use the well-established relation
\begin{equation}\label{eq:support sd 2}
\sd \sigma_C(x)=\set{z\in \clco C}{x\in\ncone{z}{\clco C}},
\end{equation}
where $C\subset \bE$ is nonempty and convex.


\begin{proposition}[The normal cone to $\Omega(A,B)$]\label{prop:ncone Omega}
	Let $\Omega(A,B)$ be as given by \cref{eq:omega} and let
	$(Y, W)\in\Omega(A,B)$. Then 
	\[\ncone{Y, W}{\Omega(A,B)}=
	\set{(X,V)\in\bE}
	{\begin{aligned}
	V\in\cK_A,\ \ip{V}{\half YY^T+W}=0\\ \AND \rge (X-VY)\subset(\ker A)^\perp
	\end{aligned}}
	\]
\end{proposition}
\begin{proof}
	Observe that $\Omega(A,B)=C_1\cap C_2\subset \bE$ where 
	\[
	C_1:=\set{Y\in \R^{n\times m}}{AY=B}\times \Sn\AND
	C_2:=\set{(Y,W)}{F(Y,W)\in \cK_A^\circ},
	\]
	with $F(Y,W):=\half YY^T+W$. Clearly, $C_1$ is affine, hence convex, and $C_2$ is also convex, which can be seen by an analogous 
	reasoning as for the convexity of $\Omega(A,B)$ (cf. the proof of Theorem \ref{thm:clco D}). Therefore,  
	\cite[Corollary 23.8.1]{RTR70} tells us that 
	\begin{equation}\label{eq:ncone sum}
	\ncone{Y,W}{\Omega(A,B)}=\ncone{Y,W}{C_1}+\ncone{Y,W}{C_2},
	\end{equation}
	where
	\[
	\ncone{Y,W}{C_1}=\set{R\in \R^{n\times m}}{\rge R\subset (\ker A)^\perp}
	\times \{0\}.
	\] 
		We now compute $\ncone{(Y,W)}{C_2}$. First recall that for any nonempty closed
	convex cone $C\subset\cE$, we have $\ncone{x}{C}=\set{z\in C^\circ}{ \ip{z}{x}=0}$ for all $x\in C$.
	Next, note that
	\[
	\nabla F(Y,W)^*U=(UY,\ U)\quad(U\in \Sn),
	\]
	so that $\nabla F(Y,W)^*U=0$ if and only if $U=0$. Hence, by 
	\cite[Exercise 10.26 Part (d)]{RoW98},
	\[
	\begin{aligned}
	\ncone{Y,W}{C_2}
	&=\set{(VY,V)}{V\in \cK_A,\ \ip{V}{\half YY^T+W}=0}.
	\end{aligned}
	\]
	Therefore, by \cref{eq:ncone sum}, $\ncone{Y,W}{\Omega(A,B)}$ is given by
	\[
	\set{(X,V)}{\rge(X-VY)\subset(\ker A)^\perp,\ V\in\cK_A,\ \ip{V}{\half YY^T+W}=0},
	\]
	which proves the result.
\end{proof}

\noindent
By combining \eqref{eq:support sd 2} and Proposition \ref{prop:ncone Omega}  
we obtain a simplified representation of the subdifferential
of the support function $\sigma_\cD(A,B)$.

\begin{corollary}[The subdifferential of $\sig_{\cD(A,B)}$]\label{cor:Subdiff}
	Let $\cD(A,B)$ be as given in \cref{eq:D}. Then, for all $(X,V)\in\dom \sig_{\cD(A,B)}$
	(see \cref{eq:dom support D}) we have 
	\[
	\sd \support{X,V}{\cD(A,B)}=
	\set{(Y, W)\in\Omega(A, B)}{\begin{aligned}
		&\exists Z\in\bR^{p\times m}:\;X=VY+A^TZ,\\
		&\ip{V}{\half YY^T+W}=0\end{aligned}}.
	\]
\end{corollary}

\begin{proof}
	This follows directly from the normal cone description in Proposition  \ref{prop:ncone Omega} and the relation \eqref{eq:support sd 2}.
\end{proof}



\section{The geometry of {\boldmath${\Omega(A, B)}$}}
\label{sec:geometry}

We first compute the relative interior and the affine hull of $\Omega(A,B)$. For these purposes, we recall an established result on the 
relative interior of a convex set in a product space.

\begin{proposition}[{\cite[Theorem 6.8]{RTR70}}]\label{prop:RintProd} Let $C\subset \bE_1\times \bE_2$. For each $y\in \bE_1$ we define $C_y:=\set{z\in \bE_2}{(y,z)\in C}$ and $D:=\set{y}{C_y\neq \emptyset}$. Then 
\[
\rint C=\set{(y,z)}{y\in \rint D,\; z\in \rint C_y }.
\]
\end{proposition}

\noindent
We use this result to get a representation for the relative interior of $\Omega(A,B)$ directly, and then mimic its technique of proof to 
tackle the affine hull. 


\begin{lemma}\label{lem:AffIntersect} Let $A,B\subset \bE$ be convex with $\rint A\cap \rint B\neq \emptyset$.  Then $\aff (A\cap B)=\aff A\cap \aff B$. 
\end{lemma}
\begin{proof} The inclusion $\aff (A\cap B)\subset \aff A\cap \aff B$ is clear since the latter set is affine and contains $A\cap B$. 

For proving the reverse inclusion,    we can assume w.l.o.g. that $0\in \rint A\cap \rint B=\rint(A\cap B)$, where for the latter equality we refer to \cite[Theorem 6.5]{RTR70}. In particular we have
\begin{equation}\label{eq:AffAB}
\aff A=\R_+A,\; \aff B=\R_+B\AND \aff(A\cap B)=\R_+(A\cap B),
\end{equation}
see \eqref{eq:RintChar} and the discussion afterwards. Now, let $x\in \aff A\cap \aff B$. If $x=0$ there is nothing to prove. If $x\neq 0$, by \eqref{eq:AffAB}, we have 
$
x=\lambda a=\mu b
$
for some  $\lambda, \mu >0$ and $a\in A,b\in B$. W.l.o.g we have $\lambda>\mu$, and hence, by convexity of $B$, we have 
\[
a=\left(1-\frac{\mu}{\lambda}\right)0+\frac{\mu}{\lambda}b\in B.
\]
Therefore $x=\lambda a\in \R_+(A\cap B)=\aff (A\cap B)$, see \eqref{eq:AffAB}.
\end{proof}

\noindent
We now prove  a result analogous to Proposition \ref{prop:RintProd}. 

\begin{proposition}\label{prop:AffProd} In addition to the  assumptions of  Proposition \ref{prop:RintProd} assume that $D$ is affine. Then 
$(y,z)\in\aff C$ if and only if $y\in D$ and $z\in \aff C_y$.
\end{proposition}

\begin{proof} We imitate the proof of \cite[Theorem 6.8]{RTR70}: Let $L:(y,z)\mapsto z$. Since $D$ is assumed to be affine (hence $D=\aff D=\rint D$), we have
\begin{equation}\label{eq:LAff}
D=L(C)=L(\rint C)=L(\aff C),
\end{equation} 
where we invoke the fact that linear mappings commute with the relative interior and the affine hull, see \cite[Theorem 6.7 and p.~8]{RTR70}.

Now fix $y\in D=\rint D$ and define the affine set  $M_y:=\set{(y,z)}{z\in \bE_2}=\{y\}\times \bE_2$.  Then, by \eqref{eq:LAff},
there exists $z\in\bE_2$ such that $y=L(y,z)$ and $(y,z)\in\rint C$. Hence, $\rint M_y\cap \rint C\neq \emptyset$ and we can invoke Lemma \ref{lem:AffIntersect} to obtain
\begin{equation*}\label{eq:MainCap}
 \aff M_y\cap \aff C=\aff( M_y \cap C)=\aff (\{y\}\times C_y)=\{y\}\times \aff C_y.
 \end{equation*}
Hence, if $y\in D, z\in \aff C_y$, we have $(y,z)\in \{y\}\times \aff C_y= M_y\cap \aff C\subset \aff C$.

In turn, for $(y,z)\in C$, we have $(y,z)\in M_y\cap \aff C=\{y\}\times C_y$, hence $z\in C_y\neq\emptyset$, so $y\in D$.
\end{proof}

\noindent
We are now in a position to prove the desired result on the relative interior and the affine hull of $\Omega(A,B)$. 

\begin{proposition}\label{prop:AffRint}
	For $\Omega(A, B)$ given by \cref{eq:omega} the following hold:
	\begin{itemize}
	\item[a)] $
	\rint\Omega(A, B) = \set{(Y, W) \in \bE}{AY=B \! \!\AND\!\! \half YY^T+W\in \rint(\cK_A^\circ)}. 
	$
	\item[b)]
 $
	\aff\Omega(A, B) = \set{(Y, W) \in \bE}{AY=B \! \!\AND\!\! \half YY^T+W\in \lin\cK_A^\circ},
	$
	
where  $\lin\cK_A^\circ = \lin\set{vv^T}{v \in \ker A}$.
	
\end{itemize}
\end{proposition}

\begin{proof}We apply the format of Proposition \ref{prop:RintProd} and \ref{prop:AffProd}, respectively, for $C:=\Omega(A,B)$.  Then 
\[
D=\set{Y}{AY=B}\AND
C_y=\left\{\begin{array}{rcl}
\cK_A^\circ-\frac{1}{2}YY^T, & \text{if } &   AY=B,\\ 
\emptyset,  & \text{else}.
\end{array}\right.
 \quad(Y\in \R^{n\times m}),
\]
\begin{itemize} 
\item[a)] Apply Proposition \ref{prop:RintProd} and observe that $\rint(\cK_A^\circ-\frac{1}{2}YY^T)=\rint(\cK_A^\circ)-\frac{1}{2}YY^T$.
\item[b)] Apply Proposition \ref{prop:AffProd} and observe that $D$ is affine, and that  $\aff(\cK_A^\circ-\frac{1}{2}YY^T)=\aff (\cK_A^\circ)-\frac{1}{2}YY^T$.
\end{itemize}
\end{proof}

\noindent
As a direct consequence of Propositions \ref{prop:Snp and polar} and \ref{prop:AffRint}, we obtain the 
following result for the special case $(A,B)=(0,0)$.

\begin{corollary}
	It holds that 
	\[
	\clco\set{(Y,-\half YY^T)}{Y\in \R^{n\times m}} = \set{(Y,W)\in\bE}{W+\half YY^T\preceq 0},	
	\]
	and
	\[
	\inter\left(\clco\set{(Y,-\half YY^T)}{Y\in \R^{n\times m}}\right)= \set{(Y,W)\in\bE}{W+\half YY^T\prec 0}.
	\]
\end{corollary}
We conclude this section by giving representations for the horizon cone and polar of 
$\Omega(A,B)$.

\begin{proposition}[The polar of $\Omega(A,B)$]\label{prop:polar of Omega}
	Let $\Omega(A,B)$ be as given in \cref{eq:omega}. Then
	\[
	\Omega(A,B)^\circ=
	\set{(X,V)}{\begin{matrix}\rge\binom{X}{B}\subset\rge M(V),\  V\in\cK_A,\\ 
		\half \tr\left(\binom{X}{B}^TM(V)^\dagger\binom{X}{B}\right)\le 1
		\end{matrix}}.
	\]
	Moreover,
	\begin{align}\label{eq:Omega hzn}
	&&\Omega(A,B)^\infty&=\{0_{n\times m}\}\times \cK_A^\circ\
	\\
	\AND&&&\nonumber\\
	\label{eq:Omega polar hzn}
	&&(\Omega(A,B)^\circ)^\infty\!&=
	\set{(X,V)}{\begin{matrix}\rge\binom{X}{B}\subset\rge M(V),\  V\in\cK_A,\\ 
		\half \tr\left(\binom{X}{B}^TM(V)^\dagger\binom{X}{B}\right)\le 0
		\end{matrix}}.
	\end{align}
\end{proposition}

\begin{proof}
	Given any nonempty closed convex set $C\subset \bE$, it is easily seen that
	$C^\circ=\set{z}{\support{z}{C}\le 1}$. 
	Consequently, our expression for $\Omega(A,B)^\circ$ follows
	from \cref{eq:support D}. 
	
	To see \cref{eq:Omega hzn}, let $(Y,W)\in\Omega(A,B)$ and recall that
	$(S,T)\in \Omega(A,B)^\infty$ if and only if $(Y+tS,W+tT)\in \Omega(A,B)$
	for all $t\ge 0$. In particular, for $(S,T)\in \Omega(A,B)^\infty$, we have   $A(Y+tS)=B$ and 
	\begin{equation}\label{eq:Om hzn t squared}
	\half\left[YY^T+t(SY^T+YS^T)+\frac{t^2}{2}{SS^T}\right]+(W+tT)\in\cK_A^\circ\quad (t>0).
	\end{equation}
Consequently, $AS=0$ and, if we divide
	\cref{eq:Om hzn t squared} by $t^2$ and let $t\uparrow\infty$,
	we see that $SS^T\in \cK_A^\circ$. But $SS^T\in \cK_A$ since
	$\rge S\subset\ker A$, so we must have $S=0$. If we now divide
	\cref{eq:Om hzn t squared} by $t$ and let $t\uparrow\infty$, we find that
	$T\in\cK_A^\circ$. Hence the set on the left-hand side of 
	\cref{eq:Omega hzn} is contained in the one on the right. To see the reverse
	inclusion, simply recall that $\cK_A^\circ$ is a closed convex cone
	so that $\cK_A^\circ+\cK_A^\circ\subset\cK_A^\circ$.
	
	Finally, we show \cref{eq:Omega polar hzn}.
	Since $(0,0)\in\Omega(A,B)^\circ$, we have $(S,T)\in (\Omega(A,B)^\circ)^\infty$
	if and only if $(tS,tT)\in\Omega(A,B)^\circ$ for all $t>0$,
	or equivalently, for all $t>0$,
	\begin{gather*}
	tT\in\cK_A\AND\exists\ (Y_t,Z_t)\in \R^{n\times m}\times\R^{p\times m}\ \mbox{ s.t.}
	\ \binom{tS}{B}=M(tT)\binom{Y_t}{Z_t}\\
	\mbox{ with }\ 
	\half \tr\left(\binom{Y_t}{Z_t}^TM(tT)\binom{Y_t}{Z_t}\right)
	\le 1,
	\end{gather*}
	or equivalently, by taking $\hZ_t:=t^{-1}Z_t$,
	\begin{gather*}
	T\in\cK_A\AND\exists\ (Y_t,\hZ_t)\in \R^{n\times m}\times\R^{p\times m}\ \mbox{ s.t.}
	\ \binom{S}{B}=M(T)\binom{Y_t}{\hZ_t}\\
	\mbox{ with }\ 
	\frac{t}{2}\tr\left(\binom{Y_t}{\hZ_t}^TM(T)\binom{Y_t}{\hZ_t}\right)
	\le 1.
	\end{gather*}
	If we take $\binom{Y_t}{\hZ_t}:=M(T)^\dagger\binom{S}{B}$, we find that
	$(S,T)\in (\Omega(A,B)^\circ)^\infty$ if and only if
	\begin{gather*}
	T\in\cK_A\AND
	\frac{t}{2}\tr\left(\binom{S}{B}^TM(T)^\dagger\binom{S}{B}\right)
	\le 1\quad(t>0),
	\end{gather*}
	which proves the result.
\end{proof}

\section{{\boldmath$\sig_{\Omega(A, 0)}$} as a gauge}\label{sec:gauge}

Note that the origin is an element of $\Omega(A,B)$ if and only if $B=0$.
In this case the support function of $\Omega(A,0)$ equals the gauge
of $\Omega(A,0)^\circ$.  
Gauges are important
in a number of applications and they posses their own duality theory 
\cite{freund,FriedlanderMacedo:2016,gaugepaper}.
An explicit representation for both $\gamma_{\Omega(A,0)^\circ}$ and  $\gamma_{\Omega(A,0)}$ 
will be given in the following theorem.


\begin{theorem}[$\sig_{\cD(A,0)}$ is a gauge]\label{th:support D(A,0) is a gauge}
	Let $\Omega(A,B)$ be as given in \cref{eq:omega}. Then
	\begin{equation}\label{eq:omega polar gauge}
	\support{X,V}{\Omega(A,0)}=\gauge{X,V}{\Omega(A,0)^\circ}
	=\gam^\circ_{\Omega(A,0)}(X,V),
	\end{equation}
	and
	\begin{equation}\label{eq:omega gauge}
	\begin{aligned}
	\gauge{Y,W}{\Omega(A,0)}\!&=\!\support{Y,W}{\Omega(A,0)^\circ}
	\\ \!&=\!
	\begin{cases}
	\!\half \sig_{\mathrm{min}}^{-1}(-Y^\dagger W(Y^\dagger)^T)\!\!&\!\!\mbox{ if }
	\, \rge Y\!\subset\! \ker A\cap\rge W, W\in\cK_A^\circ,\!\\
	\!+\infty\!\!&\!\!\mbox{ else,}
	\end{cases}
	\end{aligned}
	\end{equation}
	where $\sig_{\mathrm{min}}(-Y^\dagger W(Y^\dagger)^T)$ is the smallest
	nonzero singular-value of $-Y^\dagger W(Y^\dagger)^T$ when such an eigenvalue exists
	and $+\infty$ otherwise, e.g. when $Y=0$. Here
	we interpret $\frac{1}{\infty}$ as $0$ $(0=\frac{1}{\infty})$, and so, in particular,
	$ \gauge{0,W}{\Omega(A,0)}=\indicator{W}{\cK^\circ_A}$.
\end{theorem}
\begin{proof}  
The expression \eqref{eq:omega polar gauge} follows from \cite[Theorem 14.5]{RTR70}. To show \eqref{eq:omega gauge}, first observe that
	\begin{align}\label{eq:t Omega}
	t\,\Omega(A,0)\!&=\!\set{\!(Y,W)}{AY=0\AND \half YY^T+tW\in\cK_A^\circ\!},
	\end{align}
whose straightforward proof is left to the reader.

Given $\bt\ge 0$, by \eqref{eq:t Omega}, $(Y,W)\in t\Omega(A,0)$ for all $t>\bt$ if and only if $AY=0$ and
	$\half YY^T+t W\in \cK_A^\circ$ for all $t>\bt$. 
By Proposition \ref{prop:Snp and polar} a), this is equivalent to
$AY=0$ and 
\begin{equation}\label{eq:gauge Omega}
\half YY^T+tW=P\left(\half YY^T+t W\right)P\preceq 0\quad (t>\bt),
\end{equation}
where, again, $P$ is the orthogonal projection onto $\ker A$. 
Dividing this inequality by $t$ and taking the limit as $t\uparrow\infty$ 
tells us that $W=PWP\preceq 0$.
Since $YY^T$ is positive semidefinite, inequality \cref{eq:gauge Omega} also
tells us that $\ker W\subset \ker Y^T$, i.e.  $\rge Y\subset \rge W$. 
Consequently, 
\[
\dom \gamma_{\Omega(A,0)}
\subset \set{(Y,W)}{\rge Y\!\subset\! \ker A\cap\rge W, W\in\cK_A^\circ}.
\]
Now suppose $(Y,W)\in \dom \gamma_{\Omega(A,0)}$.  Let $Y=U\Sig V^T$ be the reduced singular-value decomposition of $Y$ where $\Sig$ is an invertible diagonal matrix and $U, V$ have orthonormal columns. Since $\rge Y\subset \rge W=(\ker W)^\perp$, we know that
$U^TWU$ is negative definite, and so
$\Sig^{-1}U^TWU\Sig^{-1}$ is also negative definite. Multiplying
\eqref{eq:gauge Omega} on the left by $\Sig^{-1}U^T$ and on the right by
$U\Sig^{-1}$ gives
\[
\mu I\preceq -2\Sig^{-1}U^TWU\Sig^{-1}\quad (0<\mu\le \bar\mu),
\]
where $\bar\mu=\bar t^{-1}$. The largest $\bar \mu$ satisfying this
inequality is
\[
\sig_{\mathrm{min}}(-2Y^\dagger W(Y^\dagger)^T)=
\sig_{\mathrm{min}}(-2\Sig^{-1}U^TWU\Sig^{-1})>0,
\]
or equivalently, the smallest possible $\bar t$ in \eqref{eq:gauge Omega} 
is $1/\sig_{\mathrm{min}}(-2Y^\dagger W(Y^\dagger)^T)$, which proves the result.
\end{proof}

\section{Conclusions}
\label{sec:conclusions}
 The representation $\Omega(A,B)$ for the closed convex hull of the set $\cD(A,B)$ in Theorem \ref{thm:clco D} is a dramatic
 simplification of the one given in \cite{Bur15}. As a consequence, we also obtain simplified expressions for both
 the normal cone to $\Omega(A,B)$ and the subdifferential for generalized matrix-fractional functions in Section \ref{sec:ncone}. 
In addition, representations for several important geometric objects related to the set $\Omega(A,B)$ are computed
 in Section \ref{sec:geometry}.
 These results provide the key
 to the applications discussed in \cite{Bur15}, and open the door to the numerous further applications discussed in \cite{BGH-applications17}.
%


\section*{Acknowledgments}
The authors   thank the Department of Mathematics and Statistics at  McGill University  for supporting  this research project. 

\bibliographystyle{siamplain}
\bibliography{references}

\end{document}